\documentclass{amsart}

\usepackage{amsmath,amssymb,amsfonts,amsthm}
\usepackage[all]{xy}
\usepackage{enumerate}
\usepackage{tikz}
\usepackage{array}
\usepackage{mathrsfs}
\usepackage{csquotes}
\usepackage{hyperref}

\theoremstyle{theorem}
\newtheorem{thm}{Theorem}[section]
\newtheorem{prop}[thm]{Proposition}

\theoremstyle{remark}
\newtheorem{rem}[thm]{Remark}

\theoremstyle{definition}

\newcommand{\A}{\mathcal{A}}

\newcommand{\Q}{\mathbb{Q}}
\newcommand{\C}{\mathbb{C}}
\newcommand{\sgn}{\mathrm{sgn}}
\newcommand{\s}{\mathscr{S}}
\renewcommand{\leq}{\leqslant}
\renewcommand{\geq}{\geqslant}

\title{Purity, formality, and arrangement complements}
\author{Cl\'{e}ment Dupont}
\address{Max-Planck-Institut f\"{u}r Mathematik\\Vivatsgasse, 7 \\53111 Bonn, Germany}
\email{cdupont@mpim-bonn.mpg.de}

\begin{document}
\maketitle

\begin{abstract}
We prove a \enquote{purity implies formality} statement in the context of the rational homotopy theory of smooth complex algebraic varieties, and apply it to complements of hypersurface arrangements. In particular, we prove that the complement of a toric arrangement is formal. This is analogous to the classical formality theorem for complements of hyperplane arrangements, due to Brieskorn, and generalizes a theorem of De Concini and Procesi.
\end{abstract}

\section{Introduction}
	
	Rational homotopy theory, introduced by Quillen and Sullivan~\cite{quillenrational,sullivaninfinitesimal}, studies topological spaces via their \textit{rational models}, which are commutative differential graded algebras. The topological spaces which have a rational model with zero differential are called \textit{formal}, and constitute a particularly nice family. There is a partial notion of formality, called~$r$-formality, for~$r\geq 0$ an integer. The larger the index~$r$, the stronger the notion of~$r$-formality, with classical formality corresponding to~$r=\infty$.
	
	Let us restrict our attention to smooth complex algebraic varieties. Deligne defined~\cite{delignehodge2} a refined algebraic structure, namely a \textit{mixed Hodge structure}, on the rational cohomology groups of any smooth variety. A mixed Hodge structure consists in particular of a \textit{weight filtration}. In the smooth and compact case, this filtration is concentrated in degree $k$ on the ~$k$-th cohomology group; we say that we get a pure Hodge structure of weight~$k$ and we recover classical Hodge theory. In the general smooth case, the ~$k$-th cohomology group has weights that range from~$k$ to~$2k$. If~$U$ is a smooth variety, two extreme situations are of particular interest: when~$H^k(U)$ is pure of weight~$k$ (as in the compact case) and when~$H^k(U)$ is pure of weight~$2k$ (as far away from the compact case as possible). Our first theorem is a \enquote{purity implies formality} result, where we take advantage of these two situations to prove (partial) formality.

	\begin{thm}[see Theorem~\ref{thmhodgepartialformality}]\label{thmhodgepartialformalityintro}
	Let~$U$ be a smooth variety,~$r\geq 0$ be an integer, or~$r=\infty$, and assume that one of the following conditions is satisfied:
	\begin{enumerate}
	\item for every integer~$k\leq r+1$,~$H^k(U)$ is pure of weight~$k$;
	\item for every integer~$k\leq r$,~$H^k(U)$ is pure of weight~$2k$.
	\end{enumerate}
	Then~$U$ is~$r$-formal.
	\end{thm}
	
	
	The study of the rational homotopy theory of smooth varieties can be traced back to the seminal paper~\cite{delignegriffithsmorgansullivan} by Deligne, Griffiths, Morgan and Sullivan, whose main result is the formality of smooth compact algebraic varieties. This is proved using classical Hodge theory. Later, Morgan~\cite{morganalgebraictopology} used Deligne's mixed Hodge theory to introduce rational models for all smooth complex varieties. These models are the only ingredients of the proof of Theorem~\ref{thmhodgepartialformalityintro}.\\
	
	Next, we apply Theorem~\ref{thmhodgepartialformalityintro} to the special case of complements of hypersurface arrangements, as follows. Let~$X$ be a smooth (not necessarily compact) variety. A finite set~$\A=\{L_1,\ldots,L_l\}$ of smooth hypersurfaces of~$X$ is a \textit{hypersurface arrangement}~\cite{duponthypersurface} if around each point of~$X$ we can find a system of local coordinates in which each~$L_i$ is defined by a linear equation. The notion of a hypersurface arrangement generalizes that of a simple normal crossing divisor, for which the local linear equations are everywhere linearly independent. We denote by 
	$$U(\A)=X - L_1\cup\cdots\cup L_l$$
	the complement of the union of the hypersurfaces in~$\A$. A \textit{stratum} of~$\A$ is a connected component of an intersection~$\bigcap_{i\in I}L_i$ for some~$I\subset\{1,\ldots,l\}$. The Leray spectral sequence of the inclusion of~$U(\A)$ inside~$X$ explains how the weights on the cohomology of the strata control the weights on the cohomology of the complement~$U(\A)$. This allows us to prove, as a corollary of Theorem~\ref{thmhodgepartialformalityintro}, a \enquote{purity implies formality} theorem in the case of complements of hypersurface arrangements, where the purity assumption now concerns the cohomology of the strata.
	
	\begin{thm}[see Theorem~\ref{thmarrangementsformalpartial}]\label{thmarrangementsformalpartialintro}
	Let~$\A$ be a hypersurface arrangement inside a smooth variety~$X$, and~$r\geq 0$ be an integer, or~$r=\infty$. Assume that for every stratum~$S$ of~$\A$ and for every integer~$k$ such that~$\mathrm{codim}(S)+k\leq r$,~$H^k(S)$ is pure of weight~$2k$. Then~$U(\A)$ is~$r$-formal.
	\end{thm}
	
	Finally, a toric arrangement is a particular case of a hypersurface arrangement inside a complex torus~$X=(\C^*)^n$, where each~$L_i$ is defined by a global equation of the form~$\{z_1^{k_1}\cdots z_n^{k_n}=a\}$, with~$k_1,\ldots,k_n$ integers and~$a$ a non-zero complex number. The study of the topology of the complement~$U(\A)$, for~$\A$ a toric arrangement, was started by De Concini and Procesi~\cite{deconciniprocesitoric} as an analogue of the classical theory of hyperplane arrangements. Following their program, we prove the following corollary of Theorem~\ref{thmarrangementsformalpartialintro}.
	
	\begin{thm}[see Theorem~\ref{thmtoricformal}]\label{thmtoricformalintro}
	Let~$\A$ be a toric arrangement, then the complement~$U(\A)$ is formal.
	\end{thm}		
	
	This theorem is analogous to the classical formality theorem for complements of hyperplane arrangements, due to Brieskorn~\cite{brieskorn}, and which is also (Theorem~\ref{thmformalhyp}) a corollary of Theorem~\ref{thmarrangementsformalpartialintro}.
	
	De Concini and Procesi~\cite{deconciniprocesitoric} already proved a special case of Theorem~\ref{thmtoricformalintro}, namely the case of \emph{unimodular} toric arrangements; this was later generalized to \emph{deletion-restriction type} toric arrangements by Deshpande and Sutar~\cite{deshpandesutar}. In both cases, formality is proved by exhibiting an algebra of closed differential forms that maps bijectively to the cohomology of the complement of the arrangement, exactly as in the case of hyperplane arrangements. Our method is very different, and does not require any a priori understanding of the cohomology algebra. \\
	
	In the present article, all cohomology groups are implicitly taken with rational coefficients. We write \enquote{cdga} for \enquote{commutative differential graded algebra}, implicitly over the rationals, where commutativity is understood in the graded sense. We write \enquote{variety} for \enquote{complex algebraic variety}.\\

	Many thanks to Christin Bibby, Yohan Brunebarbe, Filippo Callegaro, Emanuele Delucchi, Priyavrat Deshpande, Luca Moci and Sergey Yuzvinsky for stimulating discussions. Many thanks to \c{S}tefan Papadima for numerous helpful comments on rational homotopy theory, and in particular for suggesting to state the results of the present article in the context of partial formality. Finally, the comments and suggestions from the anonymous referees allowed to greatly improve the presentation.

\section{Purity implies formality}

	\subsection{Mixed Hodge structures}
	
		We refer to~\cite[2.3.8]{delignehodge2} for the precise definitions. For the needs of the present article, let us just mention that a mixed Hodge structure is given in particular by a finite-dimensional rational vector space~$H$ along with an increasing filtration
		$$\cdots\subset W_{i-1}H \subset W_i H \subset W_{i+1} H \subset \cdots$$ 
		called the \textit{weight filtration}. We denote by $\mathrm{gr}_i^WH=W_iH/W_{i-1}H$ the successive quotients. A mixed Hodge structure is \textit{pure} of weight~$w$ if~$W_{w-1}H=0$ and~$W_wH=H$, i.e. if $\mathrm{gr}_i^WH=0$ for $i\neq w$. If~$H$ is a mixed Hodge structure and~$d$ is an integer, then one can construct the \textit{Tate twist}~$H(-d)$, which is a mixed Hodge structure on the same underlying vector space~$H$ for which the weight filtration is shifted by~$2d$:~$W_iH(-d)=W_{i-2d}H$.\\
	
	In~\cite{delignehodge2}, Deligne proves the existence of a functorial mixed Hodge structure on the cohomology groups of smooth varieties. For~$X$ a smooth compact variety, this is classical Hodge theory, and the mixed Hodge structure on~$H^k(X)$ is pure of weight~$k$. For~$U$ a smooth (not necessarily compact) variety, the weight filtration in~$H^k(U)$ has the form 
		$$0=W_{k-1}H^k(U)\subset W_kH^k(U)\subset\cdots \subset W_{2k-1}H^k(U)\subset W_{2k}H^k(U)=H^k(U).$$
	In other words, the weights in~$H^k(U)$ range from~$k$ to~$2k$. A basic example is the cohomology group~$H^1(\C^*)$, which is one-dimensional and pure of weight~$2$.	
	
	\subsection{Rational models and formality}
	
		Let~$r\geq 0$ be an integer, or~$r=\infty$. Let us recall that a map~$f:A\rightarrow B$ of cdga's is a \textit{$r$-quasi-isomorphism} if the induced map on cohomology~$H^i(f):H^i(A)\rightarrow H^i(B)$ is an isomorphism for~$i\leq r$ and an injection for~$i=r+1$. A \textit{rational~$r$-model} for a topological space~$U$ is a cdga that is connected to the algebra~$A^\bullet(U)$ of piecewise polynomial forms on~$U$ by a zig-zag of~$r$-quasi-isomorphisms. We say that~$U$ is \textit{$r$-formal} if it has a rational~$r$-model with zero differential. See~\cite{macinicformality} for a discussion of this notion, and~\cite[5.4]{suciutangentcone} for an application of partial formality to the study of resonance and characteristic varieties. \\
		
		The case~$r=\infty$ is of particular interest: a~$\infty$-quasi-isomorphism is just a quasi-isomorphism, a rational~$\infty$-model is simply called a \textit{rational model} and a~$\infty$-formal topological space is simply called \textit{formal}.
	
	\subsection{The Deligne-Morgan model}
	
		If a smooth variety~$U$ is given as the complement, in a smooth compact variety, of a normal crossing divisor, then Morgan used Deligne's mixed Hodge theory to define a finite-dimensional model for~$U$ that we now describe.  \\
	
		Let~$X$ be a smooth compact variety and~$D$ be a simple normal crossing divisor in~$X$; this means that in local charts on~$X$,~$D$ looks like a union of coordinate hyperplanes, and that the irreducible components of~$D$ are smooth. We label by~$D_1,\ldots,D_s$ those irreducible components. For a subset~$I\subset \{1,\ldots,s\}$, we denote by~$D_I=\bigcap_{i\in I}D_i$ the corresponding intersection, which is a disjoint union of smooth compact varieties of codimension the cardinality of~$I$. Let us put
		$$M_q^k=\bigoplus_{|I|=q-k}H^{2k-q}(D_I)(k-q).$$
		
		Because of the Tate twist~$(k-q)$,~$M_q^k$ is a pure Hodge structure of weight~$q$. \\
	
		The following theorem summarizes results from ~\cite{delignehodge2} and ~\cite{morganalgebraictopology} on the mixed Hodge structure and the rational homotopy theory of smooth algebraic varieties.
		
		\begin{thm}\label{thmmorgan}
		\begin{enumerate}
		\item The direct sum~$M^\bullet=\bigoplus_qM^\bullet_q$ has a natural structure of a cdga, with product~$M_q^k\otimes M_{q'}^{k'}\rightarrow M_{q+q'}^{k+k'}$ and differential~$M_q^k\rightarrow M_{q}^{k+1}$.
		\item We have isomorphisms~$H^k(M^\bullet_q)\cong\mathrm{gr}^W_qH^k(X-D)$ that are compatible with the product on~$M^\bullet$ and the cup-product in cohomology.
		\item The cdga~$M^\bullet$ is a rational model for~$X-D$.
		\end{enumerate}
		\end{thm}	
		
		For the sake of completeness, we now give explicit formulas for the cdga structure on~$M^\bullet$. 
		For~$I\subset J\subset\{1,\ldots,s\}$, we denote by~$\iota_J^I:D_J\hookrightarrow D_I$ the corresponding closed immersion. For~$I,I'\subset\{1,\ldots,s\}$ such that~$I\cap I'=\varnothing$, we denote by~$\sgn(I,I')\in\{\pm 1\}$ the sign of the permutation that orders~$I\cup I'$ in increasing order.\\
		
		Let us define the \emph{product}
		$$\mu:M_q^n\otimes M_{q'}^{n'}\rightarrow M_{q+q'}^{n+n'}$$
		on elements~$x\in H^{2n-q}(D_I)$ and~$x'\in H^{2n'-q'}(D'_{I'})$. By definition,~$\mu(x\otimes x')$ is zero if~$I\cap I'\neq\varnothing$, and
		$$\mu(x\otimes x')=(-1)^{(q-n)q'}\sgn(I,I') \, (\iota^I_{I\cup I'})^*(x)\cdot (\iota^{I'}_{I\cup I'})^*(x')$$
		otherwise, where~$\iota^*$ denotes the restriction morphism associated to a closed immersion~$\iota$, and~$\alpha\cdot\alpha'$ denotes the cup-product of elements~$\alpha$ and~$\alpha'$ in cohomology. \\
		
		Let us define the \emph{differential}
		$$d:M_q^n\rightarrow M_{q}^{n+1}$$
		on an element~$x\in H^{2n-q}(D_I)$. It is given by the formula
		$$d(x)=(-1)^q\sum_{i\in I}\sgn(\{i\},I-\{i\}) \, (\iota^{I-\{i\}}_I)_*(x)$$
		where~$\iota_*$ denotes the Gysin morphism associated to a closed immersion~$\iota$.
		
		\begin{rem}\label{remleray}
		As noted in~\cite[3.1.7]{delignehodge2}, the Deligne-Morgan model~$M^\bullet$ appears as the~$E_2$ page of the Leray spectral sequence of the inclusion~$X-D\hookrightarrow X$.
		\end{rem}
		
	\subsection{Purity implies formality}
		
		\begin{thm}\label{thmhodgepartialformality}
		Let~$U$ be a smooth variety,~$r\geq 0$ be an integer, or~$r=\infty$, and assume that one of the following conditions is satisfied:
		\begin{enumerate}
		\item for every integer~$k\leq r+1$,~$H^k(U)$ is pure of weight~$k$;
		\item for every integer~$k\leq r$,~$H^k(U)$ is pure of weight~$2k$.
		\end{enumerate}
		Then~$U$ is~$r$-formal.
		\end{thm}
	
		\begin{proof}
		By Nagata's embedding theorem and Hironaka's resolution of singularities, there exists a pair~$(X,D)$ such that~$X$ is a smooth compact variety,~$D$ is a simple normal crossing divisor, and~$X- D=U$. We fix such a pair and use the corresponding model~$M^\bullet$ from Theorem~\ref{thmmorgan}.
		\begin{enumerate}
		\item By assumption and by Theorem~\ref{thmmorgan} (2) we have~$H^k(M_q^\bullet)=0$ for~$q\neq k$ and~$k\leq r+1$. We note that~$M_k^{k+1}=0$ and put, for every integer~$k$,
		$$C^k=H^k(M_k^\bullet)= \mathrm{coker}\left(M_k^{k-1}\stackrel{d}{\longrightarrow}M_k^k\right).$$
		The collection~$C^\bullet$ has a natural structure of a cdga with zero differential for which the surjection~$M^\bullet\twoheadrightarrow C^\bullet$ is a morphism of cdga's. The induced map
		$$H^i(M^\bullet)\rightarrow H^i(C^\bullet)$$
		is an isomorphism for~$i\leq r+1$, hence~$M^\bullet\twoheadrightarrow C^\bullet$ is a~$r$-quasi-isomorphism. By Theorem~\ref{thmmorgan} (3), this implies that~$C^\bullet$, with zero differential, is a rational~$r$-model for~$U$, hence the result.
		
		\item By assumption and by Theorem~\ref{thmmorgan} (2) we have~$H^k(M_q^\bullet)=0$ for~$q\neq 2k$ and~$k\leq r$. We note that~$M_{2k}^{k-1}=0$ and put, for every integer~$k$,
		$$K^k=H^k(M_{2k}^\bullet)= \mathrm{ker}\left(M_{2k}^{k}\stackrel{d}{\longrightarrow}M_{2k}^{k+1}\right).$$
		The collection~$K^\bullet$ has a natural structure of a cdga with zero differential for which the injection~$K^\bullet\hookrightarrow M^\bullet$ is a morphism of cdga's. The induced map
		$$H^i(K^\bullet)\rightarrow H^i(M^\bullet)$$
		is an injection for every~$i$, and an isomorphism for~$i\leq r$, hence~$K^\bullet\hookrightarrow M^\bullet$ is a~$r$-quasi-isomorphism. By Theorem~\ref{thmmorgan} (3), this implies that~$K^\bullet$, with zero differential, is a rational~$r$-model for~$U$, hence the result.	
	
		\end{enumerate}
		\end{proof}
		
		\begin{rem}
		In Theorem~\ref{thmhodgepartialformality} (1), one cannot replace the inequality~$k\leq r+1$ by~$k\leq r$. Indeed, in~\cite[Example 10.1]{dimcapapadimasuciu}, it is shown that the configuration space of~$n\geq 3$ ordered points on an elliptic curve is not~$1$-formal; on the other hand, it is easy to show that its first cohomology group is pure of weight~$1$.
		\end{rem}
		
		\begin{rem}
		In the context of rational homotopy theory for (not necessarily smooth) projective varieties, Chataur and Cirici used a similar argument to prove~\cite[Theorem 3.3]{chataurcirici} a \enquote{purity implies formality} result analogous to Theorem~\ref{thmhodgepartialformality} (1). 
		\end{rem}
		
		For the sake of clarity, let us state separately the case~$r=\infty$ of Theorem~\ref{thmhodgepartialformality}.
			
		\begin{thm}\label{thmhodgeformality}
		Let~$U$ be a smooth variety, and assume that one of the following conditions is satisfied:
		\begin{enumerate}
		\item for every integer~$k$,~$H^k(U)$ is pure of weight~$k$;
		\item for every integer~$k$,~$H^k(U)$ is pure of weight~$2k$.
		\end{enumerate}
		Then~$U$ is formal.
		\end{thm}
		
		\begin{rem}
		Theorem~\ref{thmhodgeformality} (1) applies for~$U=X$ a smooth compact variety, in which case we recover the main result of~\cite{delignegriffithsmorgansullivan}.
		\end{rem}
		
		\begin{rem}\label{remtype}
		In the context of Theorem~\ref{thmhodgepartialformality} (2) and Theorem~\ref{thmhodgeformality} (2), the proof shows that the mixed Hodge structure on~$H^k(U)$,~$k\leq r$, can only be of type~$(k,k)$. Indeed,~$H^k(U)$ is a Hodge sub-structure of~$M_{2k}^k=\bigoplus_{|I|=k}H^0(D_I)(-k)$, which has type~$(k,k)$.
		\end{rem}	
		
		\begin{rem}
		The case~$r=1$ of Theorem~\ref{thmhodgepartialformality} (2) states that if~$H^1(U)$ is pure of weight~$2$ then~$U$ is~$1$-formal. This fact is well-known, being a consequence of~\cite[Corollary 10.3]{morganalgebraictopology}; indeed, a topological space is~$1$-formal if and only if the Malcev Lie algebra of its fundamental group is the completion of its holonomy Lie algebra.
		\end{rem}
		
		We mention an application of Theorem~\ref{thmhodgeformality}.
		
		\begin{prop}
		Let~$X$ be a smooth compact variety and~$p\in X$ be a point. Then the complement~$X-p$ is formal.
		\end{prop}
		
		\begin{proof}
		Let us write~$d$ for the complex dimension of~$X$. The long exact localization sequence 
		$$\cdots\rightarrow  H^{k-2d}(p)(-d)\rightarrow H^k(X) \rightarrow H^k(X-p) \rightarrow H^{k-2d+1}(p)(-d) \rightarrow  \cdots$$
		gives isomorphisms~$H^k(X)\cong H^k(X-p)$ for~$k\leq 2d-2$, and an exact sequence
		$$0\rightarrow H^{2d-1}(X)\rightarrow H^{2d-1}(X-p)\rightarrow H^0(p)(-d)\rightarrow H^{2d}(X)\rightarrow H^{2d}(X-p)\rightarrow 0.$$
		By definition, the morphism~$H^0(p)(-d)\rightarrow H^{2d}(X)$ is the Gysin morphism of the inclusion~$p\hookrightarrow X$, i.e. is Poincar\'{e} dual to the morphism~$H^0(X)\rightarrow H^0(p)$, which is an isomorphism. We then have~$H^{2d-1}(X)\cong H^{2d-1}(X-p)$ and~$H^{2d}(X-p)=0$. Thus, for every integer~$k$,~$H^k(X-p)$ is pure of weight~$k$ and the claim follows from Theorem~\ref{thmhodgeformality} (2).
		\end{proof}

\section{Formality of arrangement complements}

	\subsection{Hodge theory of arrangement complements}
	
		Let~$X$ be a smooth (not necessarily compact) variety. A finite set~$\A=\{L_1,\ldots,L_l\}$ of smooth hypersurfaces of~$X$ is a \textit{hypersurface arrangement}~\cite{duponthypersurface} if around each point of~$X$ we may find a system of local coordinates in which each~$L_i$ is defined by a linear equation. The notion of a hypersurface arrangement generalizes that of a simple normal crossing divisor, for which the local linear equations are everywhere linearly independent. We denote by 
		$$U(\A)=X - L_1\cup\cdots\cup L_l$$
		the complement of the union of the hypersurfaces in~$\A$.\\
		
		A \emph{stratum} of a hypersurface arrangement~$\A$ is a non-empty connected component of an intersection~$\bigcap_{i\in I}L_i$ for~$I\subset\{1,\ldots,l\}$. Each stratum is a smooth subvariety of~$X$. We write~$\s_r(\A)$ for the set of strata of~$L$ of codimension~$r$. By convention,~$\s_0(\A)$ only contains the ambient variety~$X$.\\
		
		To a stratum~$S$ of~$\A$ one canonically attaches a finite-dimensional rational vector space~$A_S(\A)$ in the following way. Choose any point~$p\in S$ and any local chart on~$X$ around~$p$ in which all the hypersurfaces~$T_i$ containing~$S$ are defined by linear equations. Then in this local chart~$\A$ is a hyperplane arrangement, and one defines~$A_S(\A)$ to be the~$S$-local component of the Orlik-Solomon algebra of this hyperplane arrangement, see~\cite[2.2]{looijenga} or~\cite[2.4]{duponthypersurface} for more details.
			
		\begin{thm}
		The Leray spectral sequence of the inclusion~$U(\A)\hookrightarrow X$ can be computed as
		\begin{equation}\label{eqleray}
		E_2^{p,q}=\bigoplus_{S\in \s_q(\A)}H^p(S)\otimes A_S(\A)(-q) \; \Longrightarrow \; H^{p+q}(U(\A))
		\end{equation}
		and is a spectral sequence in the category of mixed Hodge structures.
		\end{thm}
		
		The proof of this theorem can be found in~\cite[3]{bibbyabelian}; the same spectral sequence, shifted, appeared in~\cite[2.2]{looijenga} and~\cite[4.3]{duponthypersurface}. The fact that the Leray spectral sequence of an algebraic map is compatible with mixed Hodge structures is a consequence of Saito's formalism of mixed Hodge modules~\cite{saitomodulesdehodgepolarisables}, see~\cite[Corollary 14.14]{peterssteenbrink}.
		
		\begin{rem}
		If~$\A$ is a normal crossing divisor, then every vector space~$A_S(\A)$ is one-dimensional. If furthermore~$X$ is a smooth projective variety, then the~$E_2$ page~(\ref{eqleray}) is the Deligne-Morgan model of~$U(\A)$, see Remark~\ref{remleray}.
		\end{rem}
	
	\subsection{Formality of arrangement complements}
	
		We fix, as in the previous paragraph, a smooth (not necessarily compact) variety~$X$ and a hypersurface arrangement~$\A$ inside~$X$.
	
		\begin{thm}\label{thmarrangementsformalpartial}
		Let~$r\geq 0$ be an integer, or~$r=\infty$. Assume that for every stratum~$S$ of~$\A$ and for every integer~$k$ such that~$\mathrm{codim}(S)+k\leq r$,~$H^k(S)$ is pure of weight~$2k$. Then the following holds:
		\begin{enumerate}
		\item for every integer~$k\leq r$,~$H^k(U(\A))$ is pure of weight~$2k$;
		\item~$U(\A)$ is~$r$-formal.
		\end{enumerate}
		\end{thm}
		
		\begin{proof}
		By assumption, the term~$E_2^{p,q}$ in the Leray spectral sequence (\ref{eqleray}) is pure of weight~$2(p+q)$ for~$p+q\leq r$. Let~$L$ be the (decreasing) Leray filtration on the cohomology of~$U(\A)$, abutment of the Leray spectral sequence. By definition the term~$E_\infty^{p,q}=\mathrm{gr}_L^pH^{p+q}(U(\A))$ is a subquotient of~$E_2^{p,q}$. Since the spectral sequence is compatible with mixed Hodge structures,~$L$ is a filtration by mixed Hodge sub-structures, and each graded quotient~$\mathrm{gr}^p_LH^k(U(\A))$ is pure of weight~$2k$, for~$k\leq r$. Thus, the whole cohomology group~$H^k(U(\A))$ is pure of weight~$2k$, for~$k\leq r$. Theorem~\ref{thmhodgepartialformality} (2) then implies the formality statement.
		\end{proof}
	
		For the sake of clarity, let us state separately the case~$r=\infty$ of Theorem~\ref{thmarrangementsformalpartial}.
	
		\begin{thm}\label{thmarrangementsformal}
		Assume that for every stratum~$S$ of~$\A$ and for every integer~$k$,~$H^k(S)$ is pure of weight~$2k$. Then the following holds:
		\begin{enumerate}
		\item for every integer~$k$,~$H^k(U(\A))$ is pure of weight~$2k$;
		\item~$U(\A)$ is formal.
		\end{enumerate}
		\end{thm}
		
		\begin{rem}
		In the context of Theorem~\ref{thmarrangementsformalpartial} and~\ref{thmarrangementsformal}, Remark~\ref{remtype} implies that the mixed Hodge structure on~$H^k(U(\A))$,~$k\leq r$, can only be of type~$(k,k)$.
		\end{rem}
		
		\begin{rem}
		We may also note (although this is not necessary for our discussion) that the Leray spectral sequence (\ref{eqleray}) degenerates at~$E_2$, i.e.~$E_\infty^{p,q}=E_2^{p,q}$. Indeed, any differential~$d_r:E_r^{p,q}\rightarrow E_r^{p+r,q-r+1}$ has a pure Hodge structure of weight~$2(p+q)$ as its source and a pure Hodge structure of weight~$2(p+q+1)$ as its target; it must thus be zero.
		\end{rem}

	\subsection{Application to hyperplane arrangements}\label{parhyperplane}
	
		A \textit{hyperplane arrangement} in~$\C^n$ is a finite set 
		$$\A=\{L_1,\ldots,L_l\}$$ 
		where each~$L_i$ is a hyperplane in~$\C^n$. The \textit{complement} of the hyperplane arrangement~$\A$ is the smooth variety
		$$U(\A)=\C^n - L_1\cup\cdots \cup L_l.$$
	
		\begin{thm}\label{thmformalhyp}
		Let~$\A$ be a hyperplane arrangement. Then the following holds:
		\begin{enumerate}
		\item for every integer~$k$,~$H^k(U(\A))$ is pure of weight~$2k$;
		\item~$U(\A)$ is formal.
		\end{enumerate}
		\end{thm}		
		
		\begin{proof}
		A hyperplane arrangement~$\A$ in~$\C^n$ is trivially a hypersurface arrangement; furthermore, all strata are isomorphic to affine spaces~$\C^d$, hence their only non-zero cohomology groups are~$H^0(\C^d)\cong \Q(0)$, pure of weight~$0$. The claim then follows from Theorem~\ref{thmarrangementsformal}. 
		\end{proof}
	
		Theorem~\ref{thmformalhyp} is classical, and is often proved in a much simpler way, without needing the Deligne-Morgan model. Let~$R^\bullet(\A)$ be the subalgebra of the algebra of holomorphic differential forms on~$U(\A)$ generated by the logarithmic forms~$\frac{df}{f}$, for~$f$ the equation of a hyperplane in~$\A$. All the differential forms in~$R^\bullet(\A)$ are closed. Then~\cite[Lemme 5]{brieskorn} implies that the inclusion of~$R^\bullet(\A)$ inside the complex of differential forms on~$U(\A)$ is a quasi-isomorphism, which implies the formality of~$U(\A)$. The purity statement follows easily from the same argument, and first appeared in~\cite{lehrerladic,shapiropure,kimweights}. 
	
	\subsection{Application to toric arrangements}
	
		A \emph{toric arrangement} in~$(\mathbb{C}^*)^n$ is a finite set 
		$$\A=\{L_1,\ldots,L_l\}$$
		where each~$L_i$ is a translated codimension~$1$ subtori of~$(\mathbb{C}^*)^n$, i.e. defined by an equation of the form~$\{z_1^{k_1}\cdots z_n^{k_n}=a\}$ with~$k_1,\ldots,k_n$ integers and~$a$ a non-zero complex number.
	
		The \emph{complement} of the toric arrangement~$\A$ is the smooth variety 
		$$U(\A)=(\mathbb{C}^*)^n - L_1\cup\cdots\cup L_l.$$
		
		\begin{thm}\label{thmtoricformal}
		Let~$\A$ be a toric arrangement. Then the following holds:
		\begin{enumerate}
		\item for every integer~$k$,~$H^k(U(\A))$ is pure of weight~$2k$;
		\item~$U(\A)$ is formal.
		\end{enumerate}
		\end{thm}		
		
		\begin{proof}
		By using the exponential cover~$\exp:\C^n\rightarrow(\C^*)^n$, one sees that a toric arrangement~$\A$ in~$(\C^*)^n$ is a hypersurface arrangement. All strata are isomorphic to tori~$(\C^*)^d$; since~$H^1(\C^*)$ is pure of weight~$2$, one derives from the K\"{u}nneth formula that~$H^k((\C^*)^d)$ is pure of weight~$2k$ for all integers~$k,d$. The claim then follows from Theorem~\ref{thmarrangementsformal}. 
		\end{proof}
		
		We note that part (1) of Theorem~\ref{thmtoricformal} was already proved, in a very similar way, by Looijenga~\cite[2.2]{looijenga}.
		
		In~\cite{deconciniprocesitoric}, De Concini and Procesi proved the formality of the complement of \emph{unimodular} toric arrangements. These are the toric arrangements~$\A=\{L_1,\ldots,L_l\}$ for which all the intersections~$\bigcap_{i\in I}L_i$, for~$I\subset\{1,\ldots,l\}$, are connected. Their proof uses differential forms and goes along the same lines as the one for hyperplane arrangements that we have described in \S\ref{parhyperplane}. This was later generalized by Deshpande and Sutar~\cite{deshpandesutar} to a broader class of toric arrangements, called \emph{deletion-restriction type}. It is natural to ask if these proofs can be extended to more general toric arrangements. We hope to address this question in a future article.

\bibliographystyle{alpha}
\bibliography{biblio}

\end{document}